\documentclass{amsart}

\usepackage{graphicx} 

\usepackage[abs]{overpic}
\usepackage{here}

\newtheorem{theorem}{Theorem} 
\newtheorem*{lemma}{Lemma} 
\newtheorem{claim}{Claim}

\theoremstyle{remark}
\newtheorem{remark}{Remark}

\newcommand{\yaji}[2]{\begin{pmatrix} #1 \\ #2 \end{pmatrix}}
\newcommand{\mat}[4]{\begin{pmatrix} #1 & #2 \\ #3 & #4 \end{pmatrix}}

\begin{document}



\title{On two-bridge ribbon knots} 

\author{Sayo Horigome}
\address{Caritas Girls' Junior \& Senior High School, 4-6-1 Nakanoshima, Tama Ward, Kawasaki, Kanagawa 214-0012, Japan}

\author{Kazuhiro Ichihara}
\address{Department of Mathematics, College of Humanities and Sciences, Nihon University, 3-25-40 Sakurajosui, Setagaya Ward, Tokyo 156-8550, Japan}
\email{ichihara.kazuhiro@nihon-u.ac.jp}

\dedicatory{Dedicated to Professor Kouki Taniyama on his 60th birthday.}


\date{\today} 



\begin{abstract}
We show that a two-bridge ribbon knot $K(m^2 , m k \pm 1)$ with $m > k >0$ and $(m,k)=1$ admits a symmetric union presentation with partial knot which is a two-bridge knot $K(m,k)$. Similar descriptions for all the other two-bridge ribbon knots are also given. 
\end{abstract}

\maketitle

\section{Introduction and statement of results}

\footnotetext{This is an arXiv version. An updated and compressed version will be published.}
The slice-ribbon conjecture is one of the well-known long standing conjectures in knot theory, which states that every slice knot is ribbon. 
A knot is called \textit{slice} if it is concordant to the unknot, and is called \textit{ribbon} if it bounds a singular disk with only ribbon singularities. 
See Fig.~\ref{2}. 
\begin{figure}[H]
\begin{center}
\includegraphics[width=.3\textwidth]{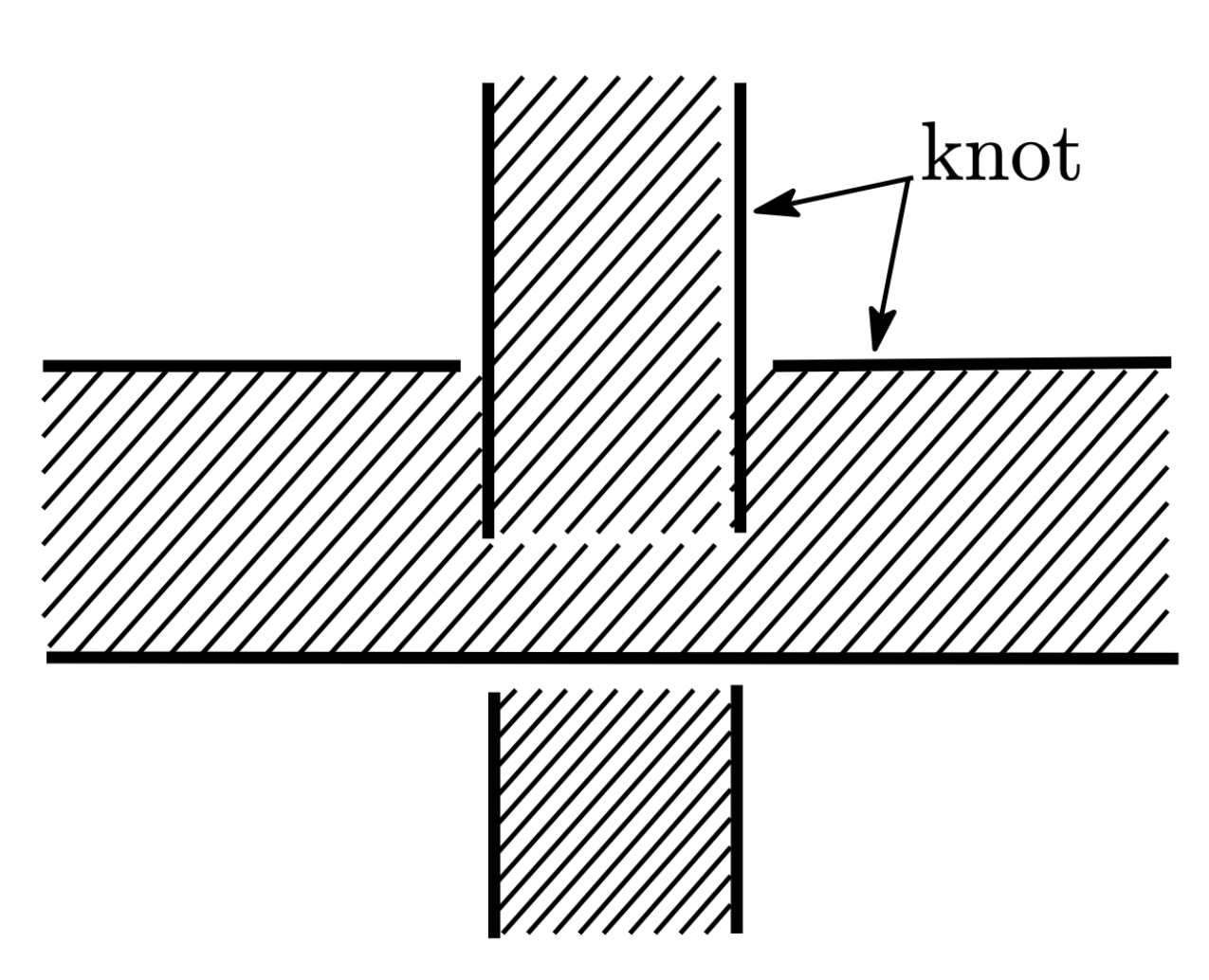}
\end{center}
\caption{ribbon singularity}\label{2}
\end{figure}
The slice-ribbon conjecture is still widely open, but in \cite{Lisca07}, Lisca proved that it holds for two-bridge knots, that is, he showed that every slice two-bridge knot is ribbon. 
A knot $K$ in the 3-space is called a \textit{two-bridge knot} if its bridge index is equal to two. 
Previously, there were 3 families of two-bridge ribbon knots given in \cite{CassonGordon86}, then in fact, Lisca showed that there are no other slice two-bridge knots. 

On the other hand, Lamm raised another conjecture regarding ribbon knots in \cite{Lamm00}. 
The conjecture states that every ribbon knot is represented as a symmetric union. 
A \textit{symmetric union} is a diagrammatic presentation for links, which was originally introduced by Kinoshita and Terasaka \cite{KinoshitaTerasaka}, and extended by Lamm \cite{Lamm00}. 
It is built from a given diagram of a knot, called the \textit{partial knot}. 
See the next section for precise definitions. 
A knot with a symmetric union presentation is visually shown to be ribbon. 
See \cite{Lamm21} for example. 
Thus the conjecture by Lamm asks whether its converse does hold or not. 
%
In fact, in \cite[Theorem 4.1]{Lamm21}, Lamm 
presented symmetric union presentations of the 3 families of two-bridge ribbon knots given in \cite{siebenmann1975exercices}. 

In this paper, we consider the symmetric union presentations and their partial knots for two-bridge ribbon knots. 

In the following, we denote by $K(p,q)$ the two-bridge knot represented by the Schubert form $S(p,q)$, or, equivalently, obtained by the closure of the rational tangle associated with a rational number $p/q$ for a positive, odd integer $p$ and an integer $q$ satisfying that $0<|q|<p$ and $(p,q)=1$. 
See \cite{Kawauchi} for more details of the definition and notation about two-bridge knots. 

In \cite[Theorem 1.2]{Lisca07}, it is proved that a two-bridge knot $K$ is ribbon if and only if $K$ is equivalent to $K(p,q)$ with $p>q>0$, $(p,q)=1$ satisfying $p=m^2$ for some odd positive $m$ and $q$ is one of the following types: 
\begin{itemize}
\item[(1)] $m k \pm 1$ with $m>k>0$ and $(m,k)=1$;
\item[(2)] $d(m \pm 1)$, where $d >1$ divides $2m \mp 1$;
\item[(3)] $d(m \pm 1)$, where $d >1$ is odd and divides $m \pm 1$.
\end{itemize}

Then, for these knots, the following were stated 
as \cite[Theorem 10.2]{lamm2006symmetric} in 
the version 1 of the preprint 
without proofs (and not included in the version 2 of the preprint and the published version \cite{Lamm21}). 

\begin{theorem}\label{MainThm}
A two-bridge knot $K(m^2 , m k \pm 1)$ with $m>k>0$ and $(m,k)=1$ is a ribbon knot which admits a symmetric union presentation with partial knot which is the two-bridge knot $K(m,k)$. 
\end{theorem}

\begin{theorem}\label{MainThm2}
A two-bridge knot $K(m^2 , d( m \pm 1))$, where $d >1$ divides $2m \mp 1$ (resp. $d >1$ is odd and divides $m \pm 1$), is a ribbon knot which admits a symmetric union presentation with a partial knot which is the two-bridge knot $K(m,d)$. 
\end{theorem}

The aim of this paper is to give detailed proofs of these theorems. 


\section*{Acknowledgments}
The first half of this paper is based on the master thesis of the first author. 
She would like to thank her supervisor Professor Shinji Fukuhara for his guidance and encouragement. 
She also thanks Mikami Hirasawa and Haruko Miyazawa for helpful discussions. 
The authors thank anonymous referee for various comments, in particular, for information about \cite{Kanenobu86}. 
They also thank Toshifumi Tanaka, Tetsuya Abe, and Christoph Lamm for useful comments.


\section{Preliminaries}

In this section, we introduce the symmetric union, which will be used to prove our theorems.


Let $D$ be a diagram of a knot $K$ and $D^*$ the diagram obtained by reflecting $D$ at an axis in the plane.\footnote{In some recent papers, the diagram obtained by reflecting $D$ is denoted by $-D$ because such a diagram presents the inverse of $K$ in the knot concordance group. In this paper, we use $D^*$ following the notation used in \cite{Lamm00, Lamm21, Tanaka24}.} 
We denote the tangle illustrated in Fig.~\ref{6} by $T(n)$ for $n \in \mathbb{Z}$ and the tangle obtained as the horizontal band with no crossings by $T(\infty )$. 

\begin{figure}[htb]
\begin{center}
\includegraphics[width=.8\textwidth]{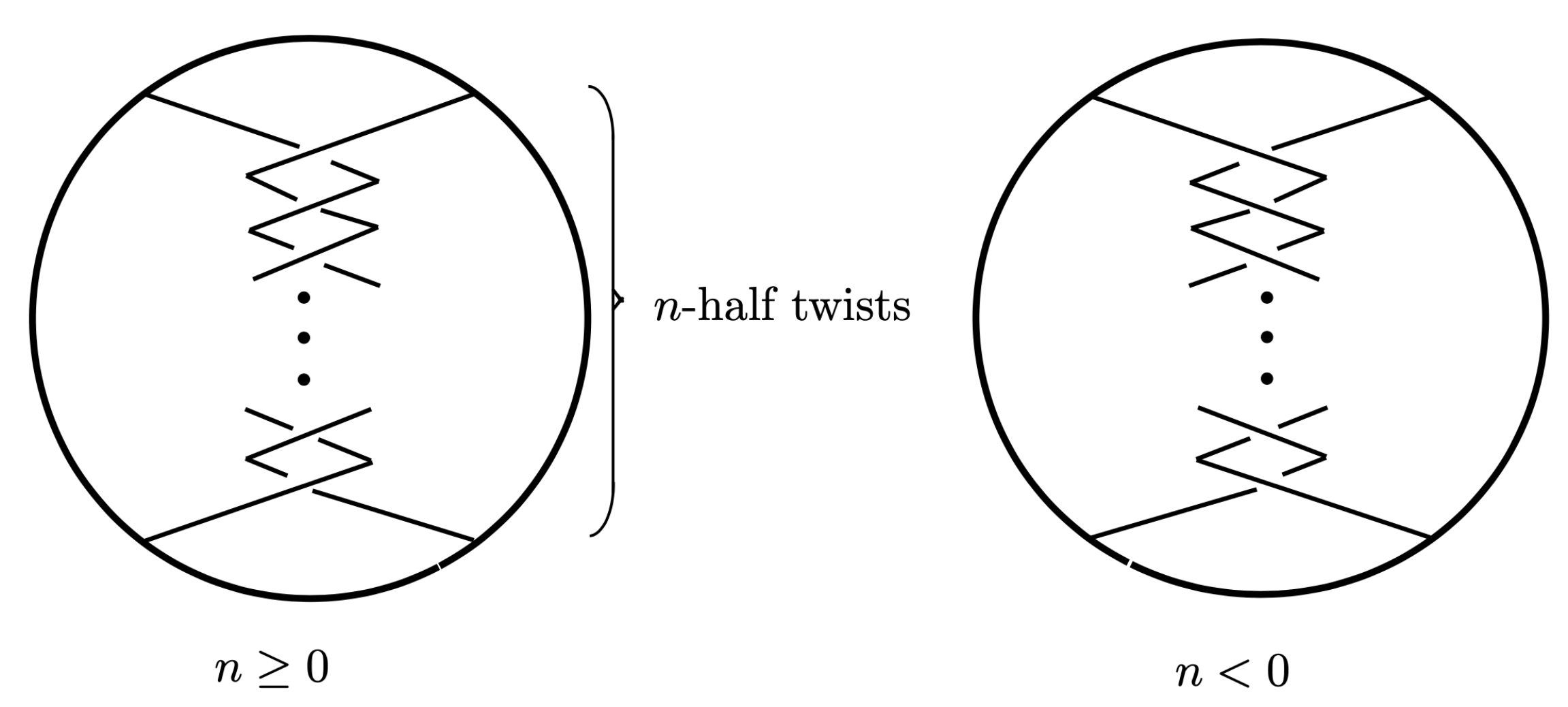}
\end{center}
\caption{Tangle $T(n)$ for $n \in \mathbb{Z}$}\label{6}
\end{figure}


Then, by using the diagrams $D$, $D^*$ and $k+1$ tangles $T(n_{0}),\ldots,T(n_{k})$ $(n_{i} \in \mathbb{Z})$, we construct a new diagram of a new knot in the following way: 
Choose $k+1$ arcs in $D$, take the corresponding arcs in $D^*$, and then, replace these $k+1$ pairs of arcs in $D$ and $D^*$ by the tangles $T(n_{0}),\ldots,T(n_{k})$ as illustrated in Fig.~\ref{7}. 

We call the diagram so obtained the \textit{symmetric union} 
and denote it by $D\cup D^{*}(n_0,n_1,\ldots ,n_k)$, 
where $n_{i}=\infty\ (i=0,\ldots,\mu -1)$ for certain $\mu \ (1 \leq \mu \leq k)$. 

\begin{figure}[htb]
\centering
\includegraphics[width=.95\textwidth]{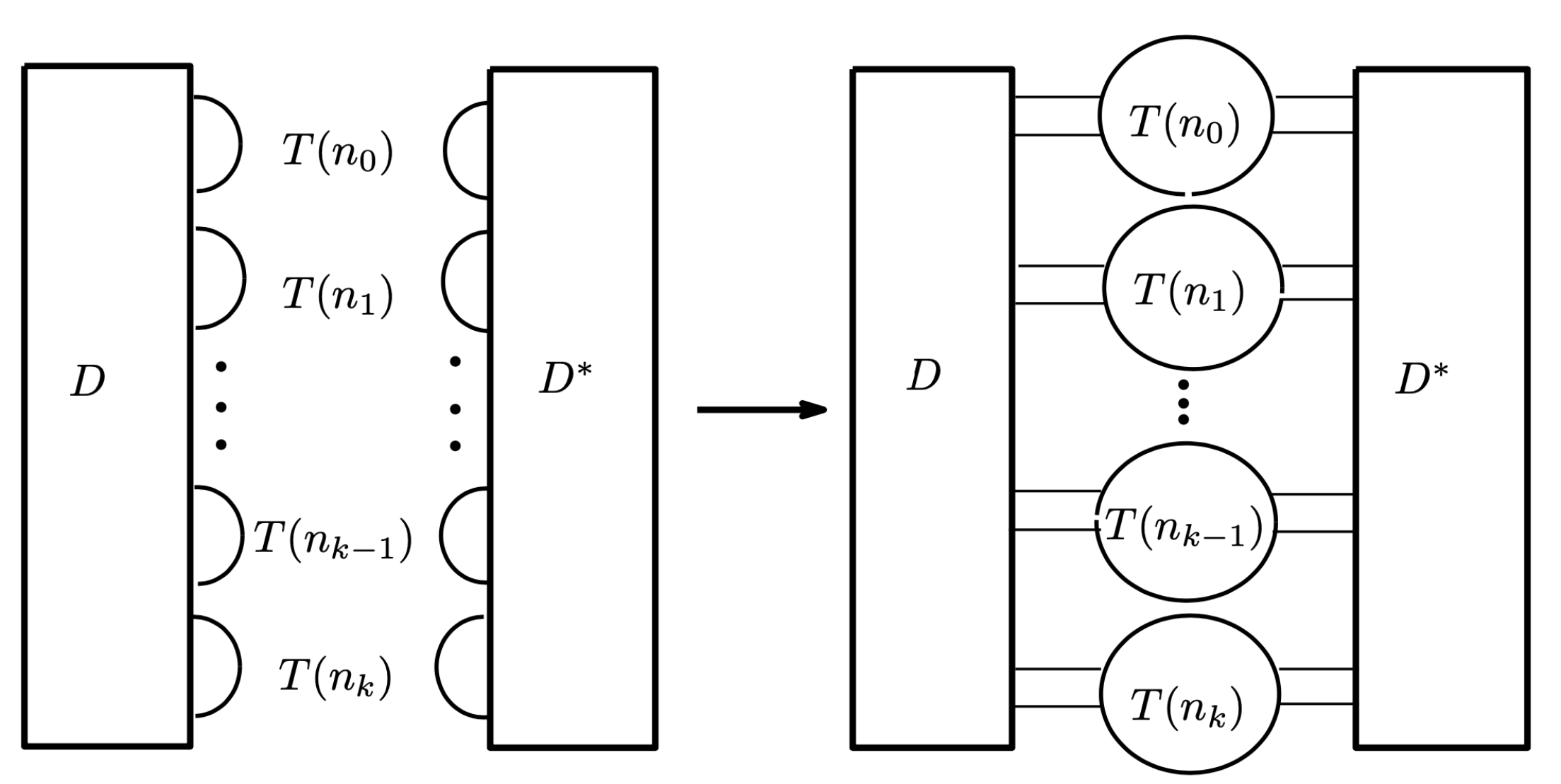}
\caption{symmetric union}\label{7}
\end{figure}

We see that any knot with a diagram which is a symmetric union is a ribbon knot. 
Moreover it was shown in \cite{Lamm00} that any knot having a symmetric union obtained from a non-trivial diagram $D$ is non-trivial. 

We have defined the symmetric union with full generality, 
but we will only consider the case that $\mu=1$ and $n_i =\pm 1$ in this paper. 


\section{Proof of Theorem~\ref{MainThm}}

In this section, we give a proof of Theorem~\ref{MainThm} by constructing symmetric union presentations for given knots. 
Throughout the rest of the paper, a continued fraction 
\[
a_1+ \dfrac{1}{a_2+ \dfrac{1}{\ddots + \dfrac{1}{a_k}}}
\]

\smallskip

\noindent
will be denoted by $[a_1, a_2, \ldots, a_k]$ for $a_{i} \in \mathbb{Z} \setminus \{ 0 \}$. 

The next well-known lemma will be used in the proof of the theorem. 
We include its elementary proof in Appendix for readers' convenience. 

\begin{lemma}\label{clm1}
Let $p,q,a_{i}$ be integers for $i=1, \ldots, n$ with $(p,q)=1$. 
Then 
$$\dfrac{p}{q}
=[a_1 ,\ldots,a_n]$$
if and only if 
$$ \yaji{p}{q} 
= \pm \mat{a_1}{1}{1}{0} \cdots \mat{a_n}{1}{1}{0} \yaji{1}{0} ,$$
where 
we regard 
$\dfrac{1}{\infty }$ as $0$ and $a_{i}+\dfrac{1}{0}$ as $\infty $. 
\end{lemma}

In the following, we will use a well-known diagram of a two-bridge knot $K(p,q)$, so-called \textit{Conway form}. 
It is drawn by using an expansion into a continued fraction $[a_1, a_2, \ldots, a_k]$ of a rational number $p/q$. 
The Conway form associated to $[a_1, a_2, \ldots, a_k]$ will be denoted by $C(a_{1},\ldots ,a_{n})$. 
See \cite{Kawauchi} in detail. 
We remark that our Conway form $C(a_{1},\ldots ,a_{n})$ is the mirror image of the diagram $L(c_1, \dots , c_n)$ associated to a continued fraction $[c_1, \dots , c_n]$ used in \cite{Lisca07}. 

\begin{figure}[htb]
\includegraphics[width=.95\textwidth]{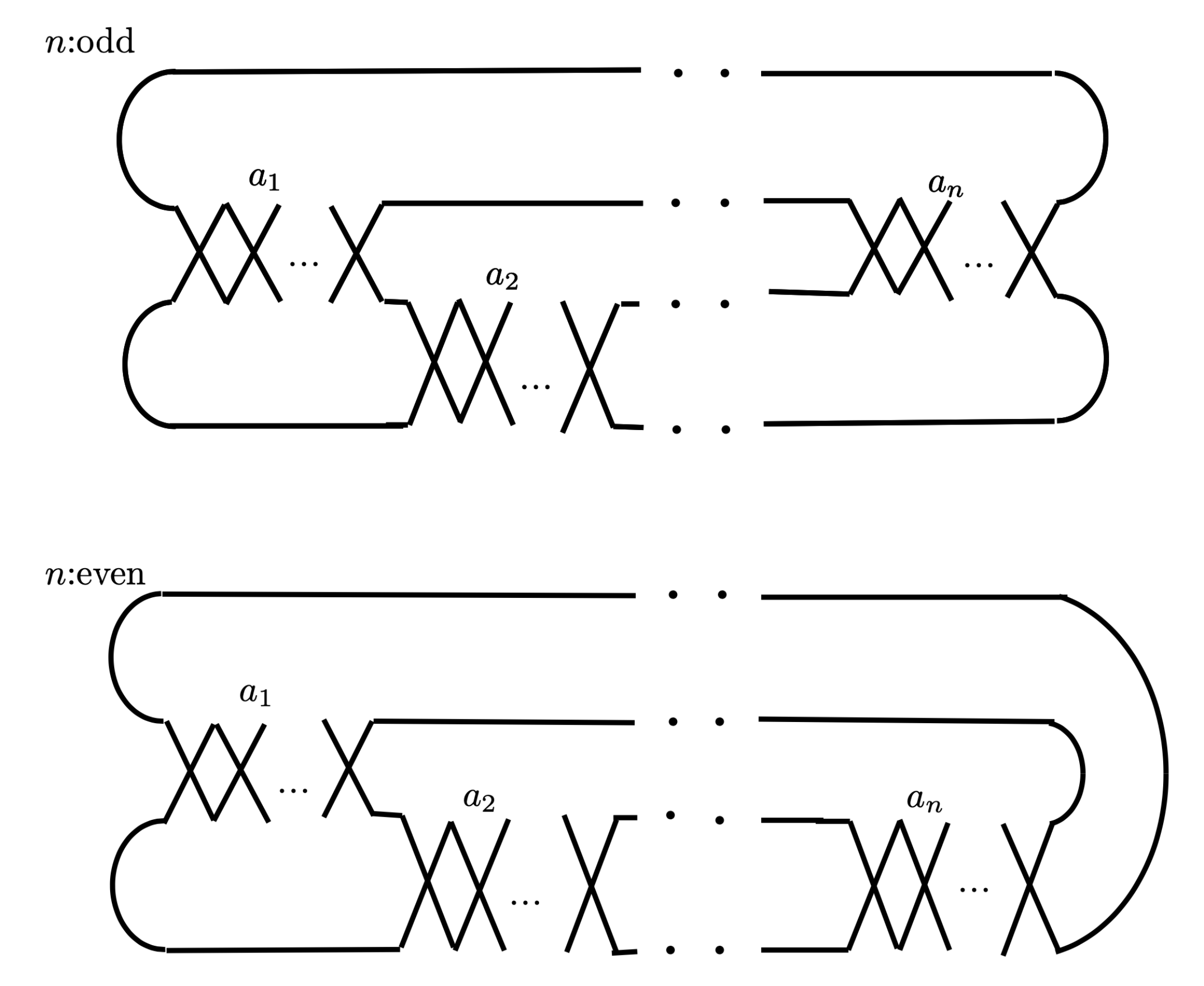}
\caption{The Conway form associated to $[a_1, \ldots, a_n]$. Here the horizontal twist associated to $a_i$ has $|a_i|$ positive (resp. negative) crossings if $a_i >0$ (resp. $a_i < 0$) with $i$ is odd or $a_i < 0$ (resp. $a_i > 0$) with $i$ is even. }\label{7.5}
\end{figure}

\begin{proof}[Proof of Theorem~\ref{MainThm}]
Let $m$ be a positive, odd integer and $k$ an integer satisfying $m > k >0$ and $(m,k)=1$. 

Suppose that $m/k$ is expanded into a continued fraction as follows: 
\begin{equation}
\frac{m}{k} =[a_1,\ldots ,a_n] \ \ \ 
(a_{i} \in \mathbb{Z} \setminus \{ 0 \}).
\label{eq3}
\end{equation}

Let $D_{0}$ be the Conway form $C(a_{1},\ldots ,a_{n})$ and 
$D_{0}^{*}$ the reflection of $D_{0}$. 
With $\varepsilon =+1$ or $-1$, 
we construct 
the symmetric union $D_{0}\cup D_{0}^{*}(\infty ,\varepsilon )$ of $D_0$ and $D_{0}^{*}$ 
as illustrated in Fig.~\ref{10} and Fig.~\ref{11}. 
Let $D$ be the diagram $D_{0}\cup D_{0}^{*}(\infty ,\varepsilon)$ so obtained. 
Then $D$ realizes the Conway form 
$C(a_{1},\ldots ,a_{n},\varepsilon ,-a_{n},\ldots ,-a_{1})$. 
Thus the knot $K$ represented by the diagram $D$ is a two-bridge knot 
which is a ribbon knot. 

\begin{figure}[H]
\begin{center}
\includegraphics[width=\textwidth]{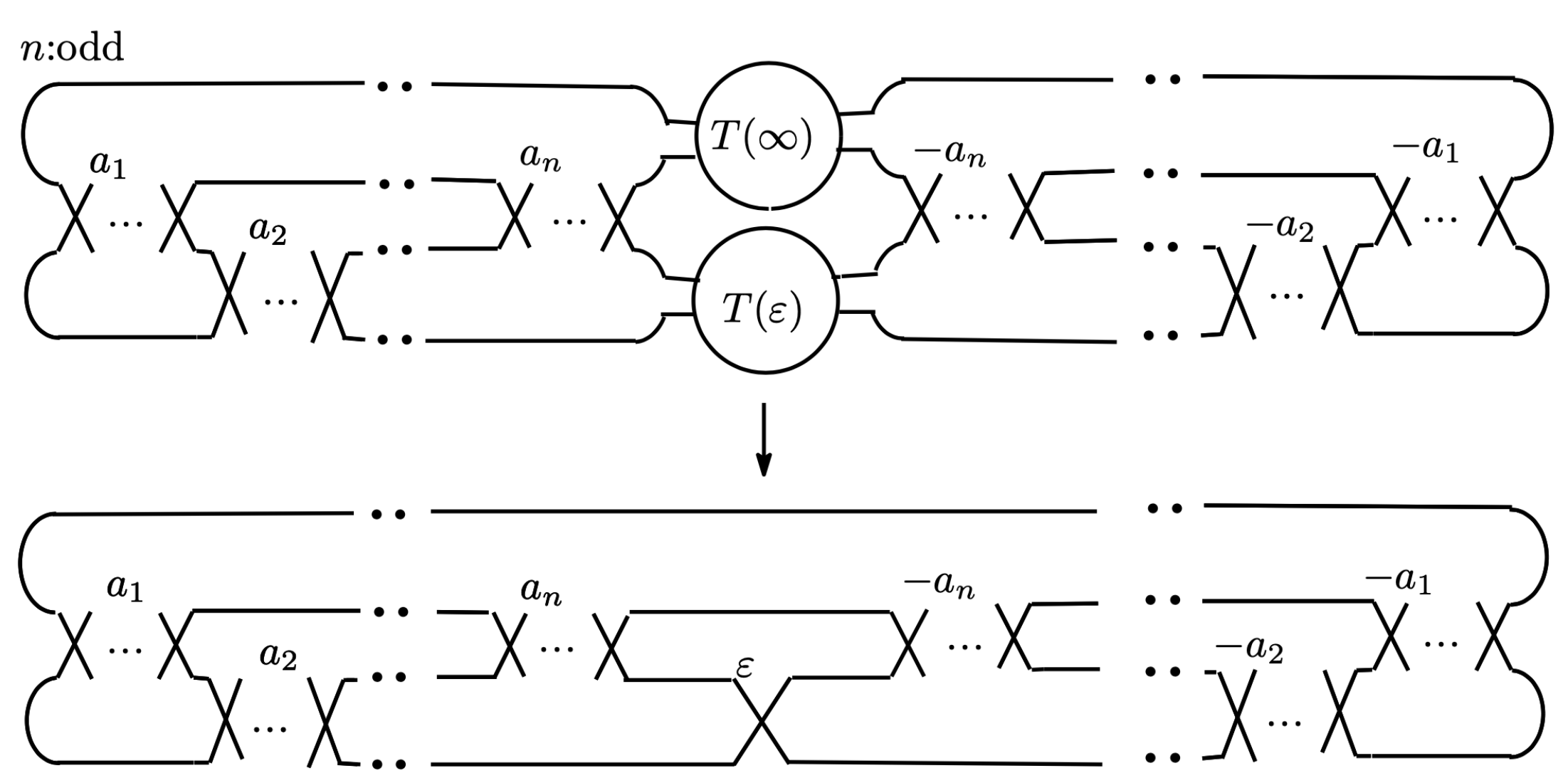}
\end{center}
\caption{$D_0\cup D_{0}^{*}(\infty , \varepsilon )$ for $n$:odd} \label{10}
\end{figure}
\begin{figure}[H]
\begin{center}
\includegraphics[width=\textwidth]{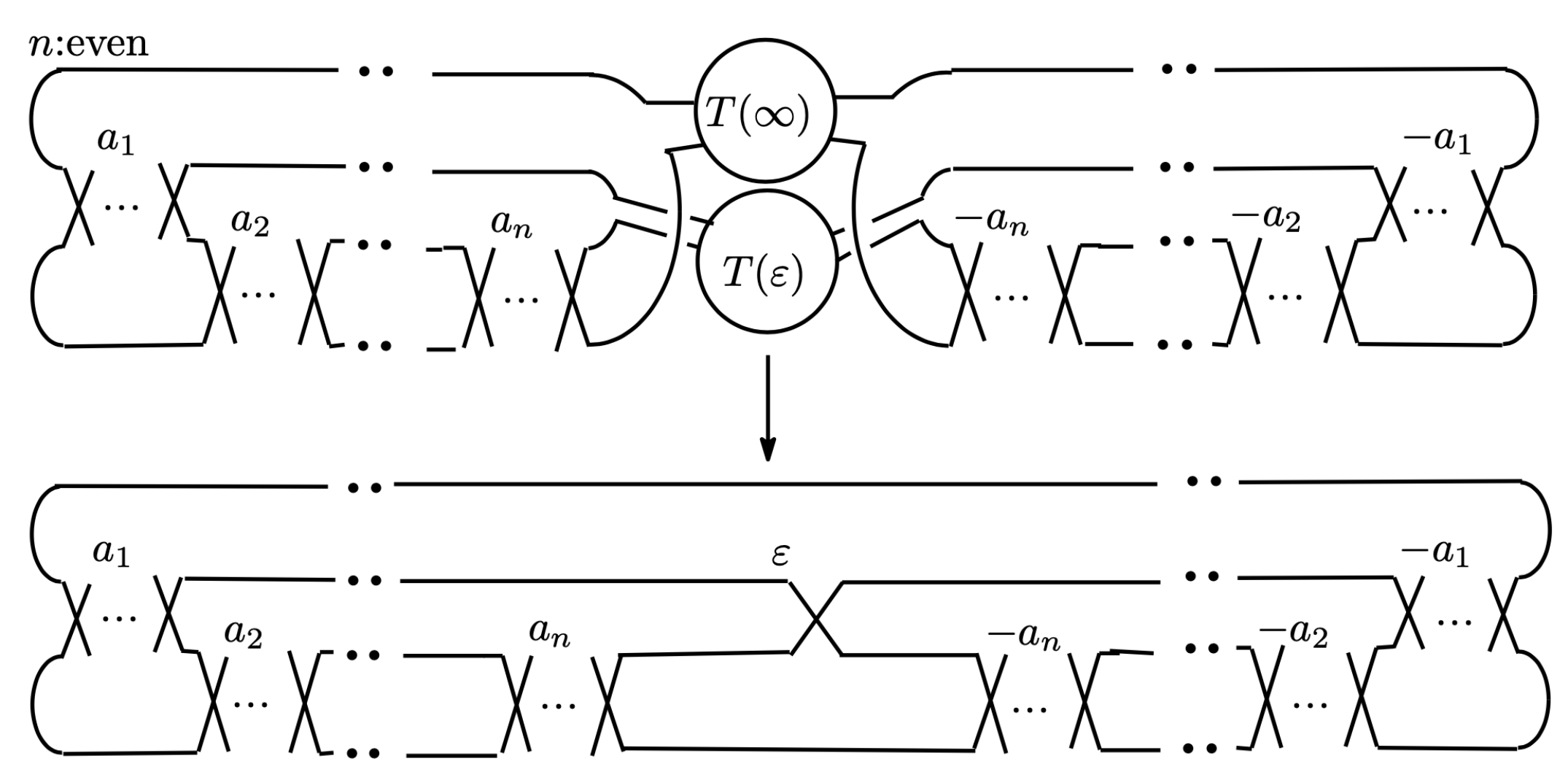}
\end{center}
\caption{$D_0\cup D_{0}^{*}(\infty , \varepsilon )$ for $n$:even} \label{11}
\end{figure}

To prove the theorem, it is sufficient to show that $K$ is equivalent to the knot $K(m^2,m k \pm 1)$, 
in other words, 
\[
\frac{m^2}{m k \pm 1} = [a_{1},\ldots ,a_{n},\varepsilon ,-a_{n},\ldots ,-a_{1}]
\]
holds. 

Suppose that a rational number $p/q$ is obtained by the continued fraction \\
$[a_{1},\ldots ,a_{n},\varepsilon,-a_{n},\ldots ,-a_{1}]$: 
That is, 
\begin{equation}
\frac{p}{q} = [a_{1},\ldots ,a_{n},\varepsilon ,-a_{n},\ldots ,-a_{1}] . \label{eq4}
\end{equation}
Since the Conway form $C(a_{1},\ldots ,a_{n},\varepsilon ,-a_{n},\ldots ,-a_{1})$ represents a two-bridge knot $K$, $p$ must be a positive, odd integer and $q$ an integer with $(p,q)=1$, and so, $K$ has the Schubert form $S(p,q)$. 
Using Eq.~\eqref{eq3} and \eqref{eq4}, we will describe $p$,$ q$ in terms of $m$, $k$. 

By Lemma~\ref{clm1}, Eq.~\eqref{eq3} implies the following. 
\begin{equation}
\yaji{m}{k} 
=\pm \mat{a_1}{1}{1}{0} \cdots \mat{a_n}{1}{1}{0} \yaji{1}{0}. \label{eq5}
\end{equation}
Also Eq.~\eqref{eq4} implies the following. 
\begin{align}
\yaji{p}{q} 
=&\pm \mat{a_1}{1}{1}{0} \cdots \mat{a_n}{1}{1}{0}\mat{\varepsilon}{1}{0}{0} \mat{-a_n}{1}{1}{0} \cdots \mat{-a_1}{1}{1}{0} \yaji{1}{0} . \label{eq6}
\end{align}


We then show: 

\begin{claim}\label{clm2}
The equations \eqref{eq5} and \eqref{eq6} implies 
\begin{equation}
\yaji{p}{q}=\pm \yaji{\varepsilon (-1)^{n} \, m^2}{\varepsilon (-1)^{n} \, m k + 1} . \label{eq7}
\end{equation}
\end{claim}

\begin{proof}
We prove the claim by induction on $n$. 

When $n=1$, Eq.~\eqref{eq5} is equal to 
\begin{equation}
\yaji{m}{k}=\pm \mat{a_1}{1}{1}{0}\yaji{1}{0}=\pm \yaji{a_1}{1} , \nonumber
\end{equation}
and so, we have $m=\pm a_1, \ k=\pm 1$. 
On the other hand, Eq.~\eqref{eq6} is equal to 
\begin{align*}
\yaji{p}{q}
&=\pm \mat{a_1}{1}{1}{0}\mat{\varepsilon }{1}{1}{0}\mat{-a_1}{1}{1}{0}\yaji{1}{0} \\
&=\pm \mat{a_1}{1}{1}{0}\yaji{-\varepsilon a_{1}+1}{-a_1} \\
&=\pm \yaji{-\varepsilon a_{1}^{2}}{-\varepsilon a_{1}+1} \\
&=\pm \yaji{\varepsilon (-1) m^2}{\varepsilon (-1) mk+1}. 
\end{align*}
Thus Eq.~\eqref{eq7} is satisfied. 

When $n=2$, Eq.~\eqref{eq5} is equal to 
\begin{equation}
\yaji{m}{k}=\pm \mat{a_1}{1}{1}{0}\mat{a_2}{1}{1}{0}\yaji{1}{0}=\pm \yaji{a_{1}a_{2}+1}{a_2} \nonumber
\end{equation}
and so, we have $m=\pm a_{1}a_{2}+1, \ k=\pm a_{2}$. 
On the other hand, Eq.~\eqref{eq6} is equal to 
\begin{align*}
\yaji{p}{q}
=&\pm \mat{a_1}{1}{1}{0}\mat{a_2}{1}{1}{0}\mat{\varepsilon }{1}{1}{0} \mat{-a_2}{1}{1}{0}\mat{-a_1}{1}{1}{0}\yaji{1}{0} \\
=&\pm \yaji{\varepsilon a_{1}^{2} a_{2}^{2}+2\varepsilon a_{1}a_{2}+\varepsilon }{\varepsilon a_{1}a_{2}^{2}+\varepsilon a_{2}+1} \\
=&\pm \yaji{\varepsilon (a_{1}a_{2}+1)^2}{\varepsilon a_{2}(a_{1}a_{2}+1)+1} =\pm \yaji{\varepsilon (-1)^{2} m^2}{\varepsilon (-1)^{2} mk+1}
\end{align*}
Thus Eq.~\eqref{eq7} is satisfied. 

Suppose that Eq.~\eqref{eq7} holds for $n \leq \ell - 1$, 
and show that it holds for $n=\ell$ for $\ell \geq 3$. 
When $n=\ell$, Eq.~\eqref{eq5} is equal to 
\begin{equation}
\yaji{m}{k} 
=\pm \mat{a_1}{1}{1}{0} \cdots \mat{a_\ell}{1}{1}{0} \yaji{1}{0} . \label{eq7.5}
\end{equation}
Then Eq.~\eqref{eq6} is regarded as a recursive formula among three terms as follows. 
{\allowdisplaybreaks
\begin{align}
\yaji{p}{q} 
=&\pm \mat{a_1}{1}{1}{0} \cdots \mat{a_\ell}{1}{1}{0}\mat{\varepsilon }{1}{0}{0} \mat{-a_\ell}{1}{1}{0} \cdots \mat{-a_1}{1}{1}{0} \yaji{1}{0} \nonumber \\
=&\pm \mat{a_1}{1}{1}{0} \cdots \mat{a_\ell}{1}{1}{0}\mat{\varepsilon }{1}{0}{0} \mat{-a_\ell}{1}{1}{0} \cdots \mat{-a_2}{1}{1}{0} \yaji{-a_{1}}{1} \nonumber \\
=&\pm \mat{a_1}{1}{1}{0} \cdots \mat{a_\ell}{1}{1}{0}\mat{\varepsilon }{1}{0}{0} \mat{-a_\ell}{1}{1}{0} \cdots \mat{-a_2}{1}{1}{0} \left( -a_{1}\yaji{1}{0}+\yaji{0}{1} \right) \nonumber \\
=&\mp a_{1}\mat{a_1}{1}{1}{0} \cdots \mat{a_\ell}{1}{1}{0}\mat{\varepsilon }{1}{0}{0} \mat{-a_\ell}{1}{1}{0} \cdots \mat{-a_2}{1}{1}{0} \yaji{1}{0} \nonumber \\
&\mbox{}\pm \mat{a_1}{1}{1}{0} \cdots \mat{a_\ell}{1}{1}{0}\mat{\varepsilon }{1}{0}{0} \mat{-a_\ell}{1}{1}{0} \cdots \mat{-a_2}{1}{1}{0} \yaji{0}{1} \nonumber \\
=&\mp a_{1}\mat{a_1}{1}{1}{0} \cdots \mat{a_\ell}{1}{1}{0}\mat{\varepsilon }{1}{0}{0} \mat{-a_\ell}{1}{1}{0} \cdots \mat{-a_2}{1}{1}{0} \yaji{1}{0} \nonumber \\
&\mbox{}\pm \mat{a_1}{1}{1}{0} \cdots \mat{a_\ell}{1}{1}{0}\mat{\varepsilon }{1}{0}{0} \mat{-a_\ell}{1}{1}{0} \cdots \mat{-a_3}{1}{1}{0} \yaji{1}{0} \label{eq8}
\end{align}}

Now, by Eq.~\eqref{eq7.5}, we have
\begin{align}
\mat{a_2}{1}{1}{0} \cdots \mat{a_\ell}{1}{1}{0} \yaji{1}{0}
&=\pm {\mat{a_1}{1}{1}{0}}^{-1} \yaji{m}{k} \nonumber \\
&=\pm \mat{0}{1}{1}{-a_1}\yaji{m}{k} \nonumber \\
&=\pm \yaji{k}{m-a_{1} k} . \nonumber 
\end{align}

This equation together with the assumption of the induction implies 
\begin{align}
\mat{a_2}{1}{1}{0}&\cdots \mat{a_\ell}{1}{1}{0}\mat{\varepsilon }{1}{1}{0}
\mat{-a_\ell}{1}{1}{0} \cdots \mat{-a_2}{1}{1}{0}\yaji{1}{0} \nonumber \\
&=\pm \yaji{\varepsilon (-1)^{\ell-1} k^2}{\varepsilon (-1)^{\ell-1} k(m-a_{1}k)+1} \label{eq12}
\end{align}

In the same way, by Eq.~\eqref{eq7.5}, we have 
\begin{align}
\mat{a_3}{1}{1}{0} \cdots \mat{a_\ell}{1}{1}{0} \yaji{1}{0}
&=\pm {\mat{a_2}{1}{1}{0}}^{-1} \yaji{k}{m-a_{1}k} \nonumber \\
&=\pm \mat{0}{1}{1}{-a_2}\yaji{k}{m-a_{1}k} \nonumber \\
&=\pm \yaji{m-a_{1}k}{k-a_{2}m+a_{1}a_{2}k} . \nonumber 
\end{align}
This equation together with the assumption of the induction implies 
\begin{align}
\mat{a_3}{1}{1}{0}&\cdots \mat{a_\ell}{1}{1}{0}\mat{\varepsilon }{1}{1}{0}\mat{-a_\ell}{1}{1}{0} \cdots \mat{-a_3}{1}{1}{0}\yaji{1}{0} \nonumber \\
&=\pm \yaji{\varepsilon (-1)^{\ell-2}(m-a_{1}k)^2}{\varepsilon (-1)^{\ell-2}(m-a_{1}k)(k-a_{2}m+a_{1}a_{2}k)+1} \label{eq14} 
\end{align}

Therefore, by the equations \eqref{eq12} and \eqref{eq14}, 
Eq.~\eqref{eq8} is equal to: 
\begin{align*}
\yaji{p}{q}=&\mp a_{1}\mat{a_1}{1}{1}{0} \mat{a_{2}}{1}{1}{0} \cdots \mat{a_\ell}{1}{1}{0}\mat{\varepsilon }{1}{0}{0} \mat{-a_\ell}{1}{1}{0} \cdots \mat{-a_2}{1}{1}{0} \yaji{1}{0} \\
&\mbox{}\pm \mat{a_1}{1}{1}{0} \mat{a_2}{1}{1}{0} \mat{a_3}{1}{1}{0}\cdots \mat{a_\ell}{1}{1}{0}\mat{\varepsilon }{1}{0}{0} \mat{-a_\ell}{1}{1}{0} \cdots \mat{-a_3}{1}{1}{0} \yaji{1}{0} \\
=&\mp a_{1} \mat{a_1}{1}{1}{0}\yaji{\varepsilon (-1)^{\ell-1} k^2}{\varepsilon (-1)^{\ell-1} k (m-a_{1}k)+1} \\
&\mbox{} \pm \mat{a_1}{1}{1}{0}\mat{a_2}{1}{1}{0}\yaji{\varepsilon (-1)^{\ell-2}(m-a_{1}k)^2}{\varepsilon (-1)^{\ell-2}(m-a_{1}k)(k-a_{2}m+a_{1}a_{2}k)+1} \\
=&\mp a_{1} \yaji{\varepsilon (-1)^{\ell-1}a_{1}k^{2}+\varepsilon (-1)^{\ell-1}k(m-a_{1}k)+1}{\varepsilon (-1)^{\ell-1}k^2} \\
&\mbox{}\pm \mat{a_{1}a_{2}+1}{a_1}{a_2}{1} \yaji{\varepsilon (-1)^{\ell-2}(m-a_{1}k)^2}{\varepsilon (-1)^{\ell-2}(m-a_{1}k)(k-a_{2}m+a_{1}a_{2}k)+1} \\
=&\mp \yaji{\varepsilon (-1)^{\ell-1}a_{1}mk+a_{1}}{\varepsilon (-1)^{\ell-1}a_{1}k^{2}} \pm \yaji{\varepsilon (-1)^{\ell-2}(m-a_{1}k)m+a_{1}}{\varepsilon (-1)^{\ell-2}(m-a_{1}k)k+1} \\
=&\pm \yaji{\varepsilon (-1)^{\ell} a_{1} m k -a_1}{\varepsilon (-1)^{\ell} a_{1} k^{2}} \pm \yaji{\varepsilon (-1)^{\ell}( m - a_{1} k) m + a_{1} }{\varepsilon (-1)^{\ell}(m-a_{1}k)k+1} \\
=&\pm \yaji{\varepsilon (-1)^{\ell} m^2}{\varepsilon (-1)^{\ell} mk+1}.
\end{align*} 

This implies that Eq.\eqref{eq7} is satisfied for $n=\ell$. 
\end{proof}

Finally, by the next claim, together with $\varepsilon = \pm 1$ and $p>0$, $p,q$ is described as $p=m^2, \ q=mk \pm 1$. 

\begin{claim}\label{clm3}
Let $m, k$ be an arbitrary pair of integers. 
Then $(m^2, mk \pm 1)=1$. 
\end{claim}

\begin{proof}
Let $\alpha \in \mathbb{Z}_+$ be the greatest common divisor of $m^2$ and $mk \pm 1$. 
Then $m^{2}$ and $mk\pm 1$ is represented as $m^2=\alpha X, \ \ mk \pm 1=\alpha Y$ with some integers $X$ and $Y$. 
Thus $mk=\alpha Y \mp 1$, and so, we obtain 
\begin{equation}
m^{2} k^{2} = \alpha^{2} Y^{2}\mp 2 \alpha Y + 1 . \label{eq10}
\end{equation}
The left hand side of Eq.~\eqref{eq10} is equal to 
$m^{2}k^{2}=\alpha X k^{2}\equiv 0 \ \pmod{\alpha}$ and 
the right hand side of Eq.~\eqref{eq10} is equal to 
$\alpha^2 Y^2\mp 2 \alpha Y + 1\equiv 1 \pmod{\alpha}$. 
Therefore, we have $0\equiv 1 \pmod{\alpha}$, which can be satisfied for the case of $\alpha =1$ only. 
We then conclude that $(m^2 , mk \pm 1)=1$. 
\end{proof}

Since $m>k>0$, we have $0<|mk \pm 1|<m^2$, so we conclude that the two-bridge knot $K$, which is also a ribbon knot has the Schubert form $S(m^2,mk \pm 1)$. 
\end{proof}

\begin{remark}
    In the reviewing process, simpler proofs of Claims~\ref{clm2} and \ref{clm3} were presented by the referee. We include the proofs in Appexdix~\ref{ap2}.
\end{remark}

\section{Proof of Theorem~\ref{MainThm2}}

In this section, we give a proof of Theorem~\ref{MainThm2}. 
The main idea comes from \cite{Lisca07} and the modifications of the diagrams essentially follow from \cite{Lamm21}. 

\begin{proof}[Proof of Theorem~\ref{MainThm2}]

We divide the proof into four claims. 

\begin{claim}\label{clm31}
A two-bridge knot $K(m^2 , d( m + 1))$, where $d >1$ divides $2m - 1$, is a ribbon knot which admits a symmetric union presentation with partial knot which is the two-bridge knot $K(m,d)$. 
\end{claim}

\begin{proof}
Since $d(m + 1) < m^2$, we have $2m - 1 > d > 1$, and $d$ must be odd because it divides $2m - 1$. 
Therefore, we can write $d = 2 s + 3$ and $ 2m - 1 = d ( 2 t + 3)$ for some $s, t \ge 0$, which implies that $m = 2 s t + 3s + 3t + 5$. 
Then we have the following, which can be verified by direct calculation. 
\[
\frac{m^2}{d(m+1)} = [ t+1 , 2 , s+1, t+1, 2 , s+1]
\]
The Conway form $C(t+1 , 2 , s+1, t+1, 2 , s+1)$ is modified as illustrated in Fig.~\ref{31} to obtain a symmetric union presentation of the knot. 
The modification of the diagram follows from \cite[Figure 10]{Lamm21}. 


\begin{figure}[H]
\centering
\begin{overpic}[width=.8\textwidth]{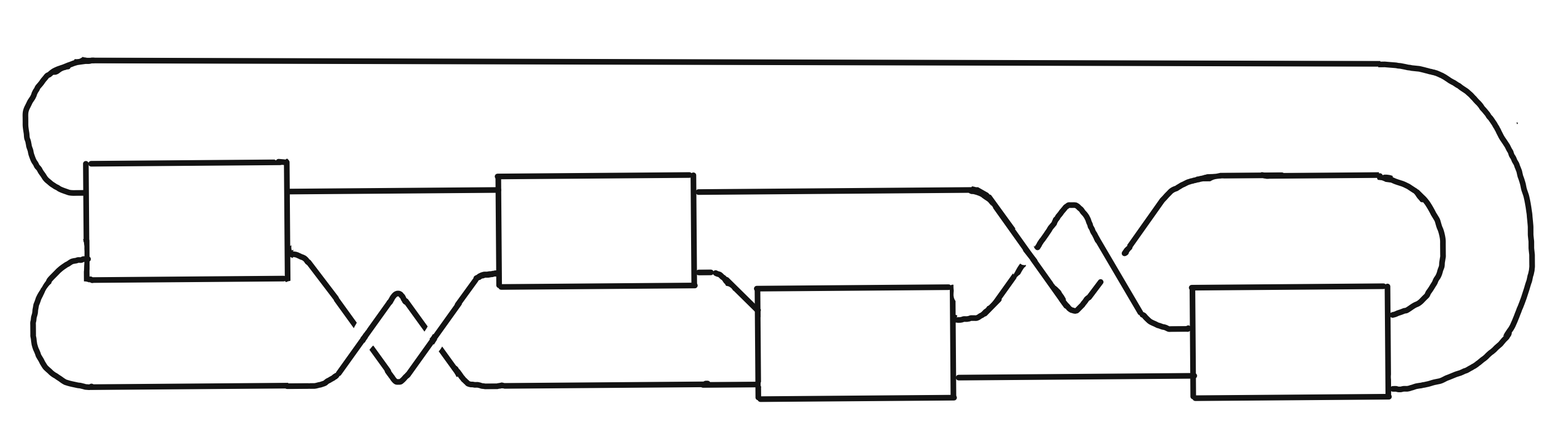}
\put(23,38){$t+1$}
\put(97,37){$s+1$}
\put(145,15){$t+1$}
\put(225,15){$s+1$}
\end{overpic}
\caption{The Conway form $C(t+1 , 2 , s+1, t+1, 2 , s+1)$. Twists in the horizontal boxes are designed in the same way as Fig.~\ref{7.5}.}
\end{figure}
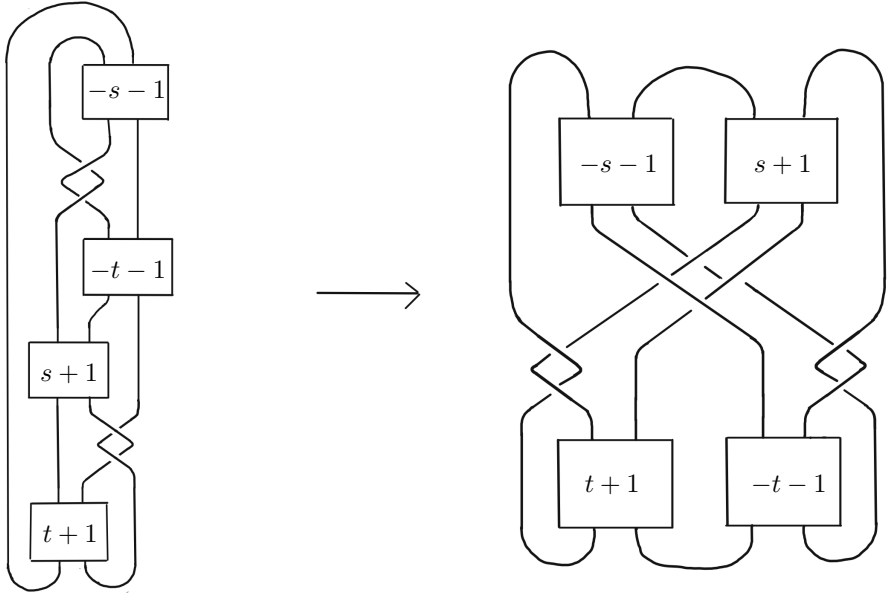
\begin{figure}[H]
\centering
\begin{overpic}[width=.95\textwidth]{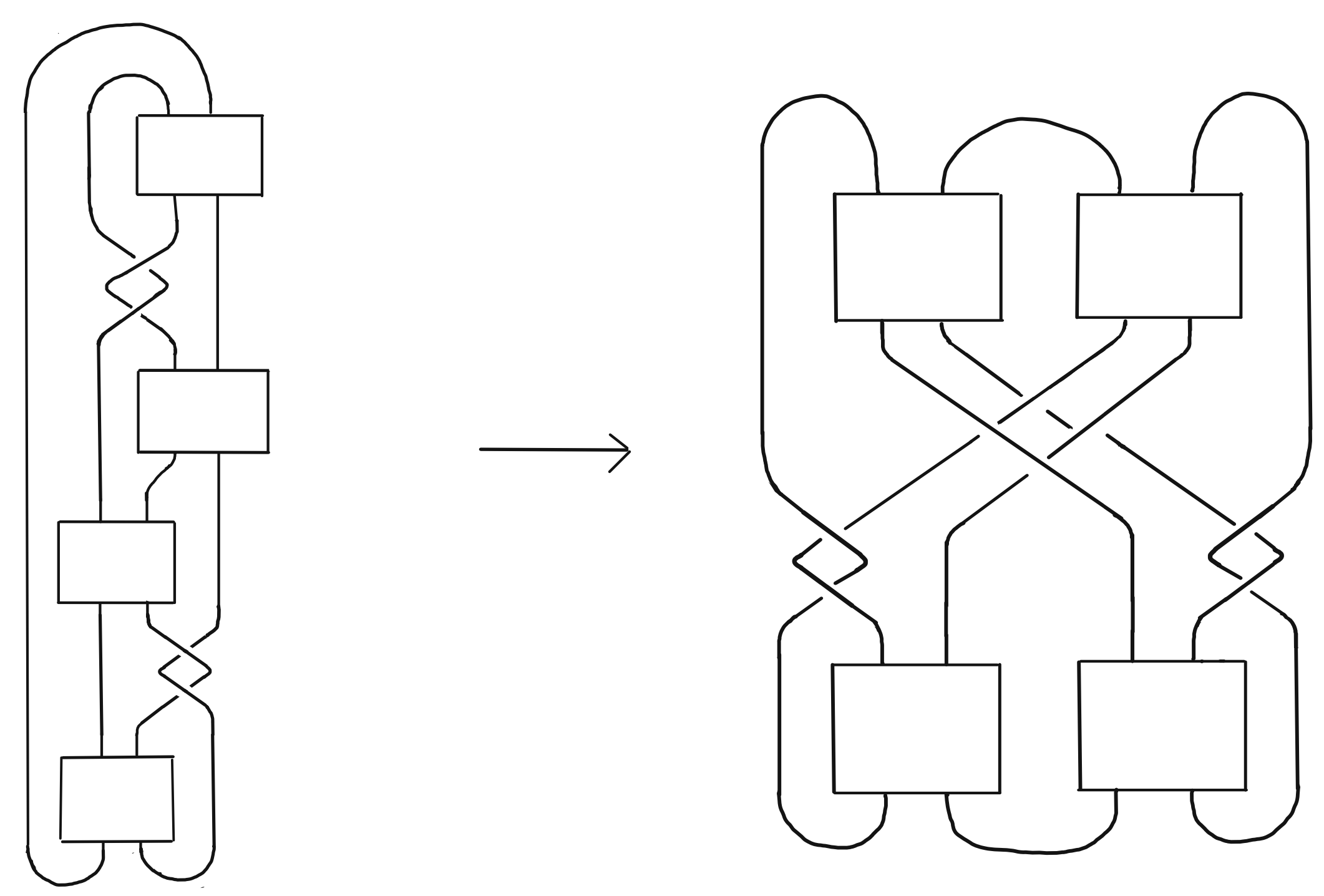}
  \put(37,189){$-s-1$}
  \put(38,121){$-t-1$}
  \put(19,82){$s+1$}
  \put(20,21){$t+1$}
  \put(223,161){$-s-1$}
  \put(289,161){$s+1$}
  \put(225,40){$t+1$}
  \put(288,40){$-t-1$}
\end{overpic}
\vspace*{8pt}
\caption{
Twists in the vertical boxes are designed in the same way as Fig.~\ref{6}.}\label{31}
\end{figure}

The partial knot of the obtained symmetric union is the two-bridge knot represented by $C(t+1,2,-s-2)$ as shown in Fig.~\ref{33}. 

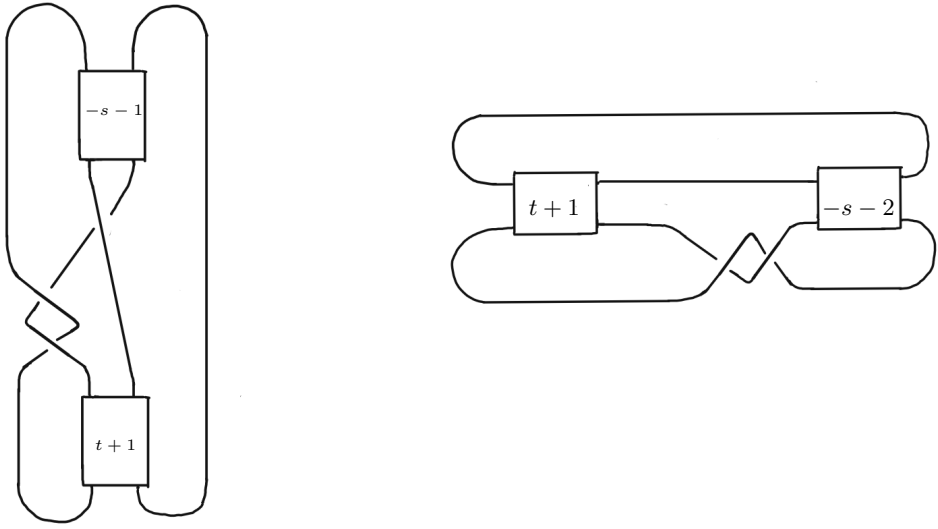
\begin{figure}[H]
\centering
\begin{overpic}[width=\textwidth]{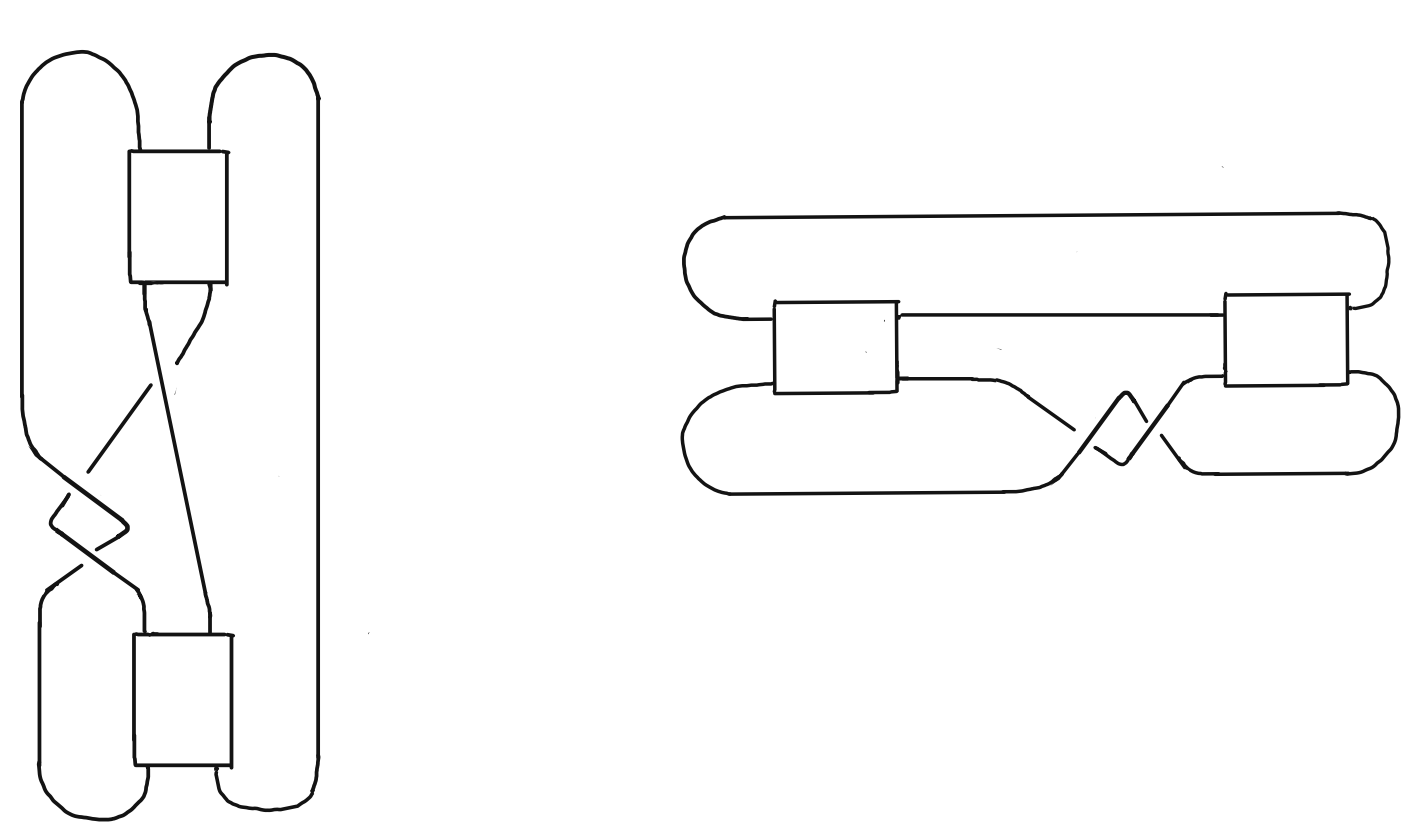}
\put(35,158){\tiny $-s-1$}
\put(39,31){\tiny $t+1$}
\put(203,120){\small $t+1$}
\put(314,120){\small $-s-2$}
\end{overpic}
\caption{The partial knot $C(t+1,2,-s-2)$. 
} \label{33}
\end{figure}

Since 
\[
[ t+1,2,-s-2 ] = \frac{2 s t + 3s + 3t + 5}{2s+3} = \frac{m}{d}, 
\]
the partial knot is the two-bridge knot $K(m,d)$. 
\end{proof}

\begin{claim}\label{Clm41}
A two-bridge knot $K(m^2 , d( m - 1))$, where $d >1$ divides $2m + 1$, is a ribbon knot which admits a symmetric union presentation with partial knot which is a two-bridge knot $K(m,d)$. 
\end{claim}

\begin{proof}
As in the proof of the previous claim, we can write $d = 2 s + 3$ and $ 2m + 1 = d ( 2 t + 3)$ for some $s, t \ge 0$, which implies that $m = 2 s t + 3s + 3t + 4$. 
Then we have the following, which can be verified by direct calculation. 
\[
\frac{m^2}{d(m-1)} = [ t+2 , -2 , -s-1, t+2, -2 , -s-1 ]
\]
The Conway form $C(t+2 , -2 , -s-1, t+2, -2 , -s-1)$ is modified as illustrated in Fig.~\ref{41} to obtain a symmetric union presentation of the knot. 
The modification of the diagram follows from \cite[Figure 10]{Lamm21}.
\begin{figure}[H]
\centering
\begin{overpic}[width=.8\textwidth]{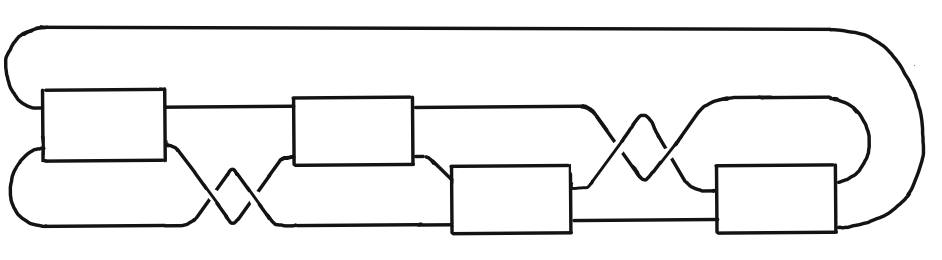}
\put(23,38){$t+2$}
\put(94,38){$-s-1$}
\put(145,15){$t+2$}
\put(223,15){$-s-1$}
\end{overpic}
\caption{The Conway form $C(t+2 , -2 , -s-1, t+2, -2 , -s-1)$}
\end{figure}
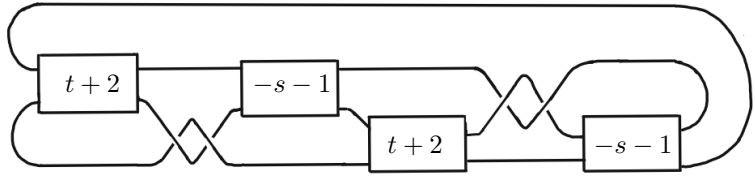
\begin{figure}[H]
\centering
\begin{overpic}[width=.95\textwidth]{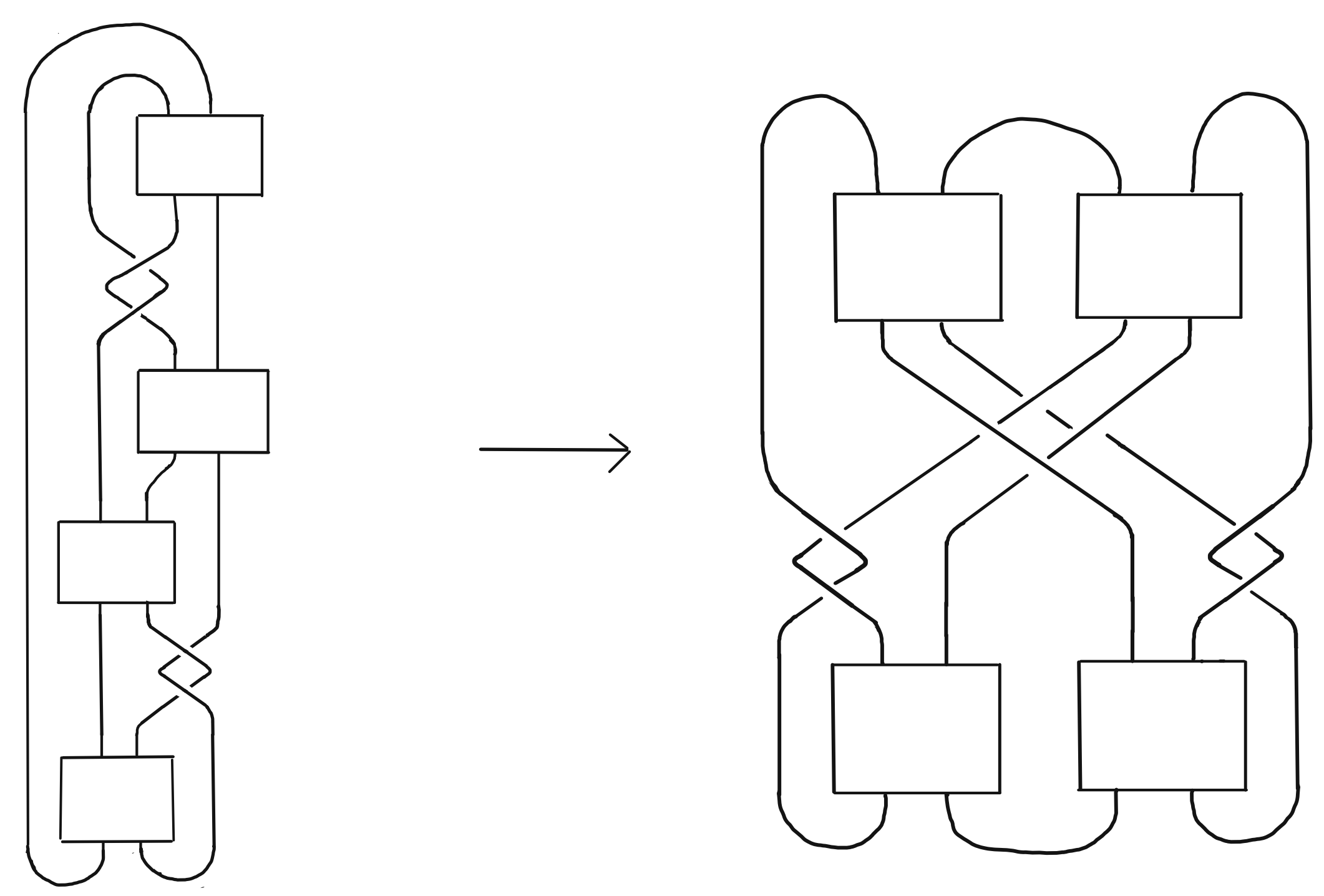}
  \put(42,187){$t+2$}
  \put(38,122){$-s-1$}
  \put(15.5,81){$-t-2$}
  \put(20,22){$s+1$}
  \put(225,161){$t+2$}
  \put(285,161){$-t-2$}
  \put(225,39){$s+1$}
  \put(285,39){$-s-1$}
\end{overpic}
\caption{The modification of the diagram.} 
\label{41}
\end{figure}
The partial knot of the obtained symmetric union is the two-bridge knot represented by $C(s+1,2,t+1)$ as shown in Fig.~\ref{42}. 
\begin{figure}[H]
\centering
\begin{overpic}[width=\textwidth]{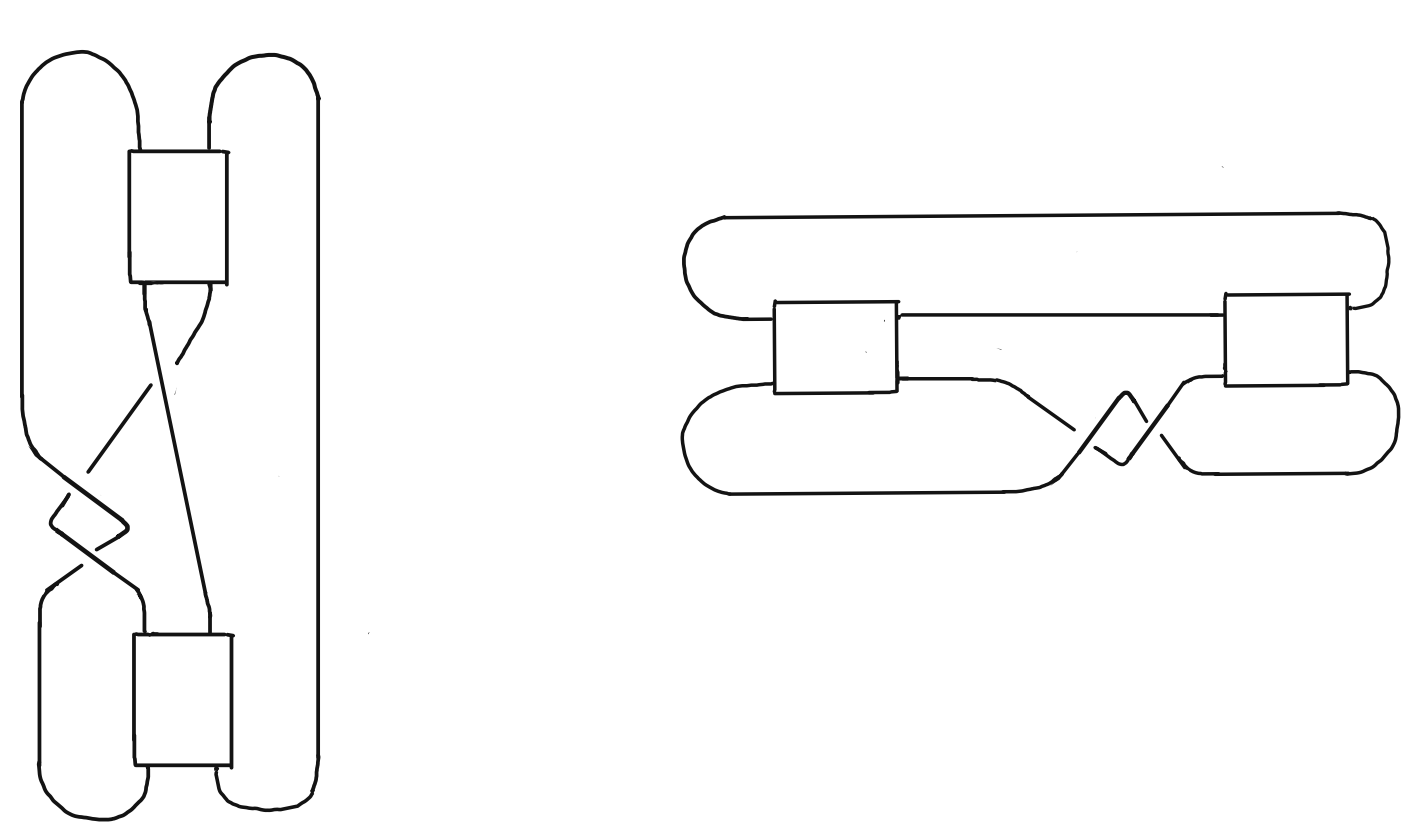}
\put(37,158){\scriptsize $t+2$}
\put(38,31){\scriptsize $s+1$}
\put(202,122){\small $s+1$}
\put(318,122){\small $t+1$}
\end{overpic}
\caption{The partial knot $C(s+1,2,t+1)$} \label{42}
\end{figure} 
Since we have 
\[
[ s+1,2,t+1 ] = \frac{2 s t + 3s + 3t + 4}{2s+3} = \frac{m}{d}, 
\]
the partial knot is the two-bridge knot $K(m,d)$. 
\end{proof}

\begin{claim}
A two-bridge knot $K(m^2 , d( m + 1))$, where $d >1$ is odd and divides $m + 1$, is a ribbon knot which admits a symmetric union presentation with partial knot which is a two-bridge knot $K(m, d)$.
\end{claim}

\begin{proof}
Since $d(m + 1) < m^2$, we have $m + 1 > d > 1$, we can write $d = 2 s + 3$ and $ m + 1 = d ( t + 2)$ for some $s, t \ge 0$, which implies that $m = 2 s t + 4s + 3t + 5$. 
Then we have the following, which can be verified by direct calculation. 
\[
\frac{m^2}{d(m+1)} = [ t+2, -s-1, -2, t+2, 2 , s+1 ]
\]
The Conway form $C(t+2, -s-1, -2, t+2, 2 , s+1)$ is modified as illustrated in Fig.~\ref{51} to obtain a symmetric union presentation of the knot. 
The modification of the diagram is similar to \cite[Figure 9]{Lamm21}.
\begin{figure}[H]
\centering
\begin{overpic}[width=.8\textwidth]{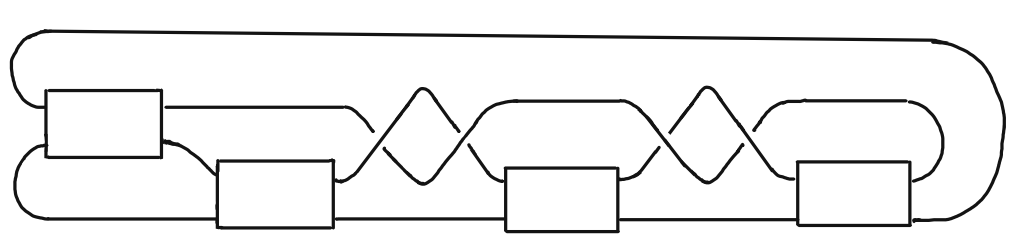}
\put(18,33){$t+2$}
\put(63,13){$-s-1$}
\put(148,12){$t+2$}
\put(230,13){$s+1$}
\end{overpic}
\caption{The Conway form $C(t+2, -s-1, -2, t+2, 2 , s+1)$}
\end{figure}
\begin{figure}[H]
\centering
\begin{overpic}[width=\textwidth]{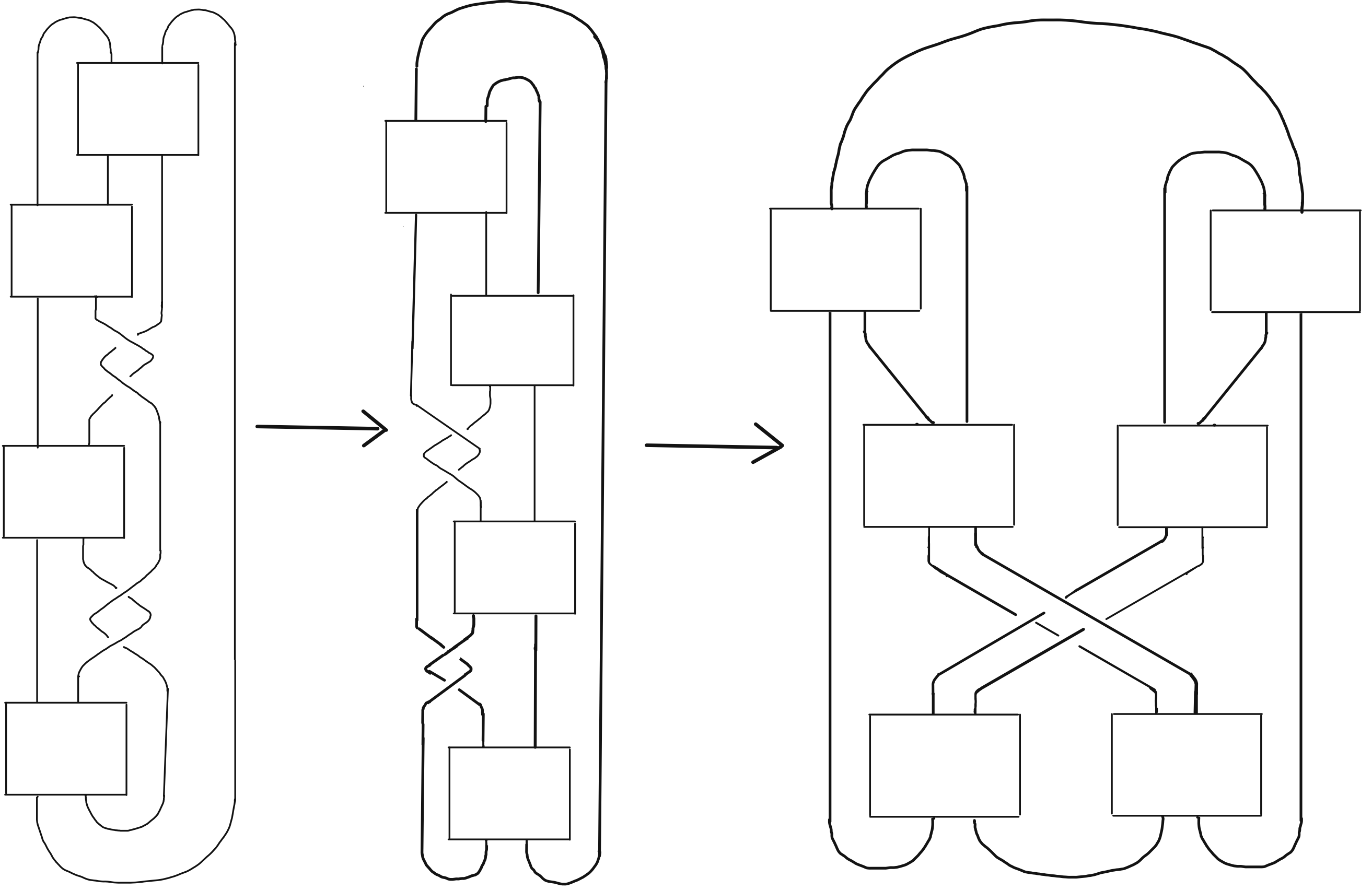}
 \put(25,202){$t+2$}
 \put(8,164){$s+1$}
 \put(1.5,100){$-t-2$}
 \put(2,34){$-s-1$}
 \put(105,186){$t+2$}
 \put(125,141){$s+1$}
 \put(120,80){$-t-2$}
 \put(119,21){$-s-1$}
 \put(210,164){$t+2$}
 \put(320,161){$-t-2$}
 \put(235,105){$s+1$}
 \put(298,105){$-s-1$}
 \put(235,28){$s+1$}
 \put(296,28){$-s-1$}
\end{overpic}
\caption{The modification of the diagram.} \label{51}
\end{figure}
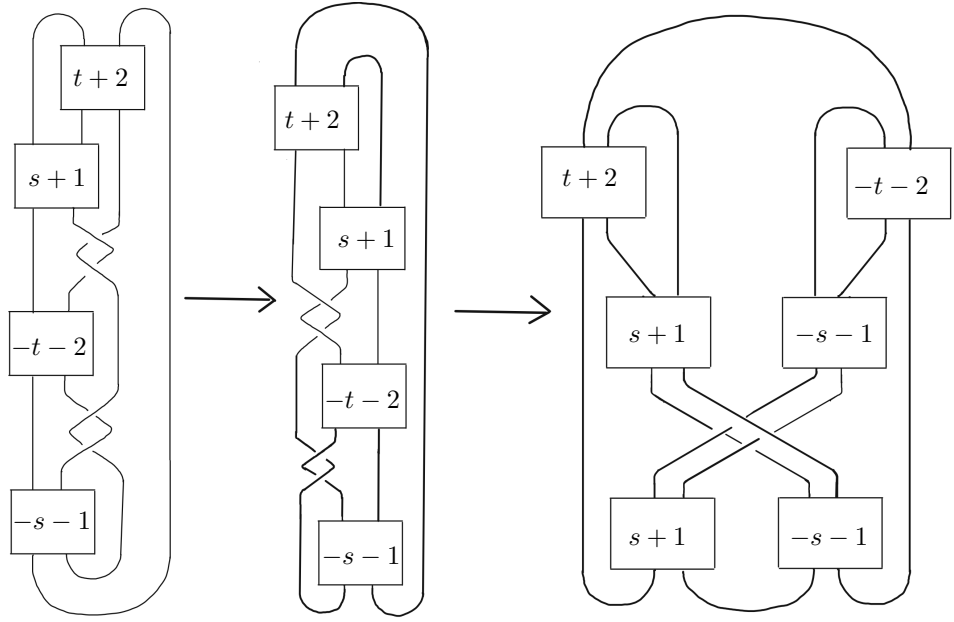
The partial knot of the obtained symmetric union is the two-bridge knot represented by $C(t+2,-2s-3)$ as shown in Fig.~\ref{53}. 
\begin{figure}[H]
\centering
\begin{overpic}[width=.6\textwidth]{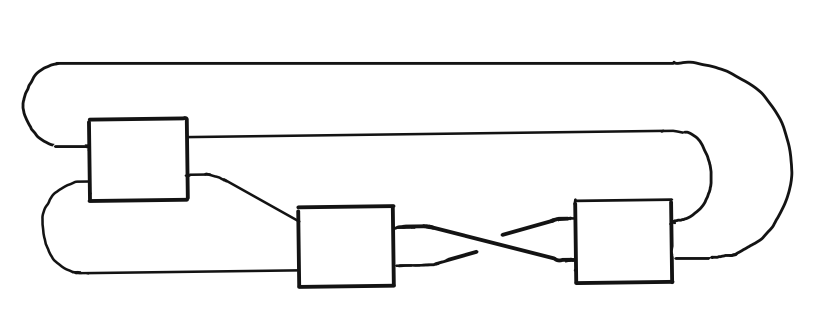}
\put(28,40){\scriptsize $t+2$}
\put(81,17){\tiny $-s-1$}
\put(154,18){\tiny $-s-1$}
\end{overpic}
\caption{The partial knot $C(t+2,-2s-3)$} \label{53}
\end{figure} 
Since we have 
\[
[ t+2, -2s-3 ] = \frac{2 s t + 4s + 3t + 5}{2s+3} = \frac{m}{d}, 
\]
the partial knot is the two-bridge knot $K(m,d)$.
\end{proof}

\begin{claim}
A two-bridge knot $K(m^2 , d( m - 1))$, where $d >1$ is odd and divides $m - 1$, is a ribbon knot which admits a symmetric union presentation with partial knot which is a two-bridge knot $K(m, d)$). 
\end{claim}

\begin{proof}
As in the proof of the previous claim, we can write $d = 2 s + 3$ and $ m - 1 = d ( t + 1)$ for some $s, t \ge 0$, which implies that $m = 2 s t + 2s + 3t + 4$. 
Then we have the following, which can be verified by direct calculation. 
\[
\frac{m^2}{d(m-1)} = [ t+1, s+1, 2, t+1, -2 , -s-1 ]
\]
The Conway form $C(t+1, s+1, 2, t+1, -2 , -s-1)$ is modified as illustrated in Fig.~\ref{61} to obtain a symmetric union presentation of the knot. 
The modification of the diagram follows from \cite[Figure 9]{Lamm21}.
\begin{figure}[H]
\centering
\begin{overpic}[width=.8\textwidth]{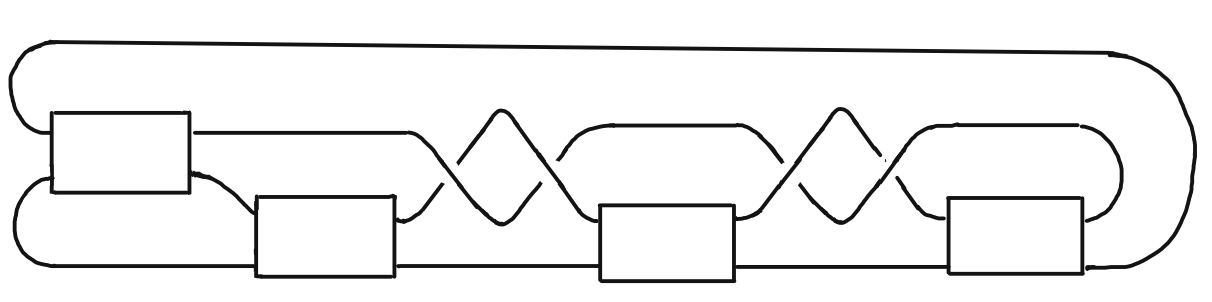}
\put(18,33){$t+1$}
\put(65,13){$s+1$}
\put(148,12){$t+1$}
\put(226,13){$-s-1$}
\end{overpic}
\caption{The Conway form $C(t+1, s+1, 2, t+1, -2 , -s-1)$}
\end{figure}
\begin{figure}[H]
\centering
\begin{overpic}[width=\textwidth]{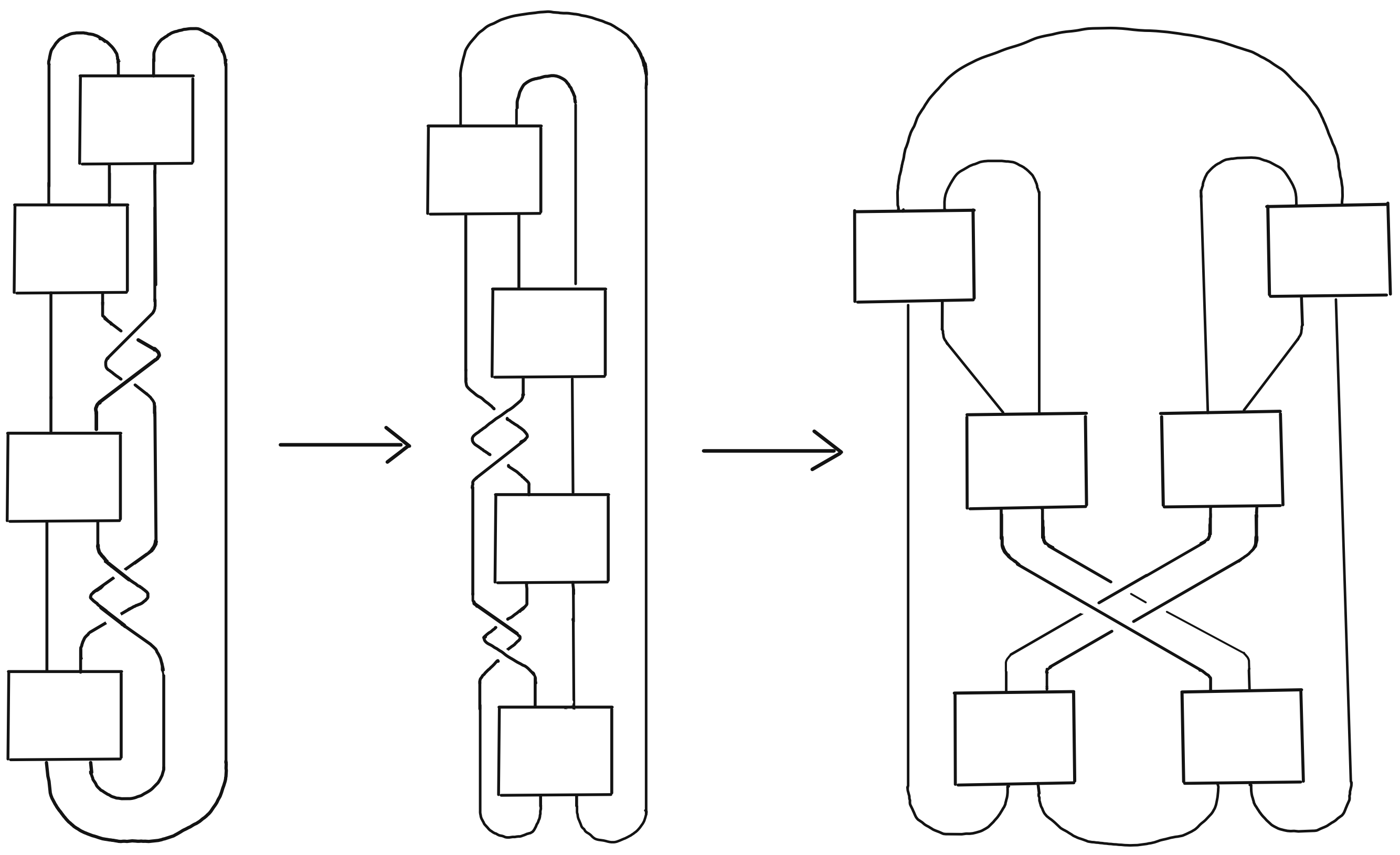}
 \put(25,185){$t+1$}
 \put(3,154){$-s-1$}
 \put(2,95){$-t-1$}
 \put(5,32){$s+1$}
 \put(115,173){$t+1$}
 \put(126,130){$-s-1$}
 \put(126,76){$-t-1$}
 \put(130,24){$s+1$}
 \put(225,150){$t+1$}
 \put(328,150){$-t-1$}
 \put(249.5,97){$-s-1$}
 \put(304,97){$s+1$}
 \put(246.5,26){$-s-1$}
 \put(309,27){$s+1$}
\end{overpic}
\caption{The modification of the diagram.} \label{61}
\end{figure}
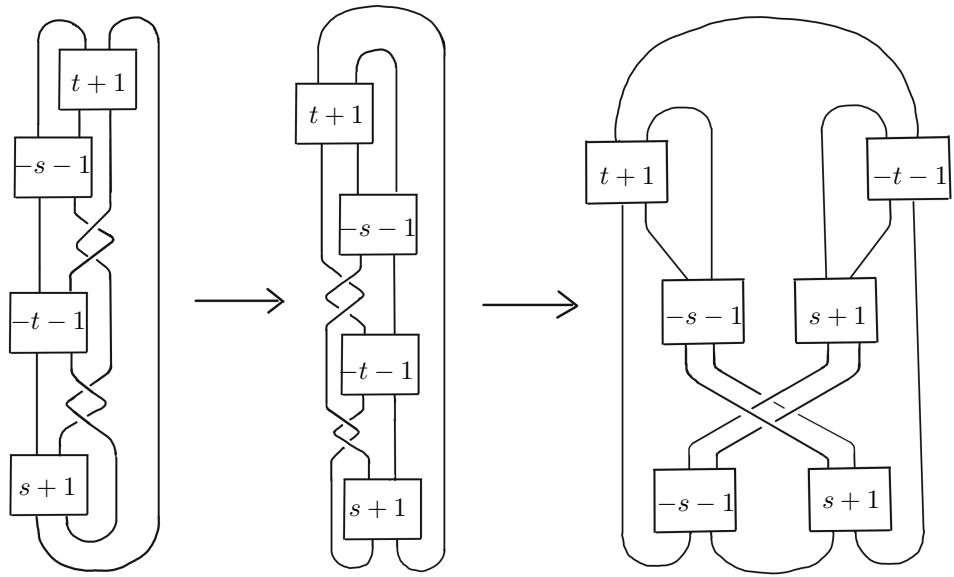
The partial knot of the obtained symmetric union is the two-bridge knot represented by $C(t+1, 2s+3)$ as shown in Fig.~\ref{63}. 
\begin{figure}[H]
\centering
\begin{overpic}[width=.6\textwidth]{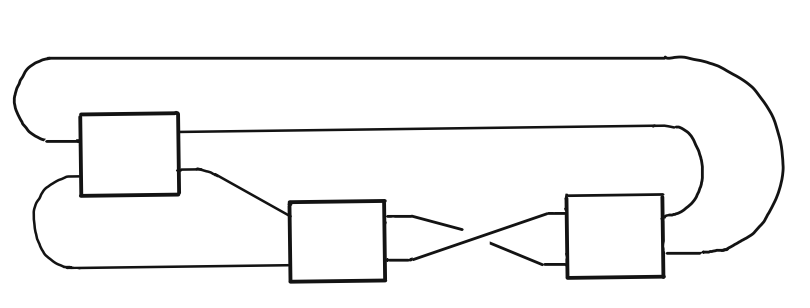}
\put(26,38){\scriptsize $t+1$}
\put(81,15){\scriptsize $s+1$}
\put(155,17){\scriptsize $s+1$}
\end{overpic}
\caption{The partial knot $C(t+1, 2s+3)$} \label{63}
\end{figure}
Since we have 
\[
[ t+1, 2s+3 ] = \frac{2 s t + 2s + 3t + 4}{2s+3} = \frac{m}{d}, 
\]
the partial knot is the two-bridge knot $K(m,d)$.
\end{proof}

The proof of Theorem~\ref{MainThm2} is now complete. 
\end{proof}

\begin{remark}
    The fractions $\frac{m^2}{d(m+1)}$ would have even continued fraction expansions of length 6, and so, the associated two-bridge knots are of genus 3. 
    On the other hand, our modifications of the diagrams in the proofs of the claims can be applied for non-even continued fractions and also for $d <-1$. 
    For example, for $[t+1,2, s+1, t+1,2, s+1]$ in Claim \ref{clm31}, when $s=t=0$, we see $\frac{25}{18}=[1,2,1,1,2,1]$. 
    From $[1,2,1,1,2,1]$, we have a symmetric union presentation of the knot $K(25,18)$ with a partial knot associated to $\frac{5}{3}=[1,2,-2]$. 
    Also we have $\frac{25}{18}=[2,-2, 2, 2,-2, 2]$. 
    The continued fraction $[2,-2, 2, 2,-2, 2]$ appears in Claim~\ref{Clm41} as $[t + 2, -2, -s-1, t+2, -2, -s-1]$ with $(s,t)=(-3,0)$ and $(m,d)=(-5,-3)$. 
    In this case, a symmetric union presentation of the knot $K(25,18)$ with a partial knot associated to $\frac{5}{3}=[-2,2,1]$ is obtained. 
\end{remark}


\appendix

\section{}

We here include an elementary proof of the following lemma. 

\begin{lemma}
Let $p,q,a_{i}$ be integers for $i=1, \ldots, n$ with $(p,q)=1$. 
Then 
$$\dfrac{p}{q}
=[a_1 ,\ldots,a_n]$$
if and only if 
$$ \yaji{p}{q} 
= \pm \mat{a_1}{1}{1}{0} \cdots \mat{a_n}{1}{1}{0} \yaji{1}{0} ,$$
where 
we regard 
$\dfrac{1}{\infty }$ as $0$ and $a_{i}+\dfrac{1}{0}$ as $\infty $. 
\end{lemma}

\begin{proof}
We prove the claim by induction on $n$. 

First we show the if part. 
When $n=1$, we have $\dfrac{p}{q}=a_1=\dfrac{a_1}{1}$ and 
\begin{equation*}
\yaji{p}{q}=\pm \mat{a_1}{1}{1}{0} \yaji{1}{0}=\pm \yaji{a_1}{1} ,
\end{equation*}
and so, the assertion holds. 

Assume that it holds for $n \leq k-1$ with $k \geq 2$. 
Let $p', q'$ denote coprime integers defined by 
\begin{equation}
\dfrac{p'}{q'} = [a_2 ,\ldots,a_k]. \label{eq1}
\end{equation}
Then, by the assumption of the induction, we have 
\begin{equation*}
\yaji{p'}{q'}=\pm \mat{a_2}{1}{1}{0} \cdots \mat{a_k}{1}{1}{0} \yaji{1}{0}. 
\end{equation*}

When $n=k$, by Eq.~\eqref{eq1}, we obtain 
\begin{align*}
\dfrac{p}{q}
&=[a_1,a_2,\ldots ,a_k] = a_1+ \dfrac{1}{\dfrac{p'}{q'}}
= \dfrac{a_{1}p'+q'}{p'}. 
\end{align*}
Here, by the assumption of the induction 
together with the fact that $a_{1}p'+q'$ and $p'$ are coprime integers, 
the following holds for $p,q$: 
\begin{align*} 
\yaji{p}{q}
&=\pm \yaji{a_{1}p'+q'}{p'} = \pm \mat{a_1}{1}{1}{0} \yaji{p'}{q'}\\
&=\pm \mat{a_1}{1}{1}{0} \mat{a_2}{1}{1}{0} \cdots \mat{a_k}{1}{1}{0} \yaji{1}{0}.
\end{align*}
Therefore the assertion holds for $n=k$. 

Next we show the only if part. 
When $n=1$, by 
\begin{equation*}
\yaji{p}{q} = \pm \mat{a_1}{1}{1}{0} \yaji{1}{0} = \pm \yaji{a_1}{1}, 
\end{equation*}
we have $p= \pm a_1, \ q= \pm 1 \ $. 
Thus $\dfrac{p}{q}=a_1$, and the assertion holds. 

Assume that it holds for $n \leq k-1$ with $k \geq 2$. 
Let $p', q'$ denote coprime integers defined by 
\begin{equation}
\yaji{p'}{q'}=\pm \mat{a_2}{1}{1}{0} \cdots \mat{a_k}{1}{1}{0} \yaji{1}{0}. \label{eq2}
\end{equation}
Then, by the assumption of the induction, we have 
\begin{equation*}
\dfrac{p'}{q'}=[a_2 ,\ldots ,a_k]. 
\end{equation*}

When $n=k$, by Eq.~\eqref{eq2}, we obtain 
\begin{align*}
\yaji{p}{q}
&=\pm \mat{a_1}{1}{1}{0}\mat{a_2}{1}{1}{0} \cdots \mat{a_k}{1}{1}{0} \yaji{1}{0}\\
&=\pm \mat{a_1}{1}{1}{0} \yaji{p'}{q'}\\
&=\pm \yaji{a_{1}p'+q'}{p'}. 
\end{align*}
Here, by the assumption of the induction togther with the fact that $a_{1}p'+q'$ and $p'$ are coprime integers, 
$p=\pm (a_{1}p'+q')$ and $q=\pm p'$ for $p,q$. 

Consequently we have 
\begin{align*}
\dfrac{p}{q}
&=\dfrac{a_{1}p'+q'}{p'}= a_{1}+\dfrac{q'}{p'}\\
&=a_{1}+\dfrac{1}{\dfrac{p'}{q'}}= a_1+ \dfrac{1}{a_2+ \dfrac{1}{\ddots + \dfrac{1}{a_k}}}
\end{align*}
and so, the assertion holds for $n=k$. 
\end{proof}

 \section{}\label{ap2}

In the reviewing process, simple proofs of Claims~\ref{clm2} and \ref{clm3} are presented by the referee. 
We include them here. 

We first prepare the following lemma. 
This is also well-known and can be proved in the same way as in the previous section. 

\begin{lemma}\label{lemB}
Let $p$, $q$, $a_{i}$ be integers for $i=1, \ldots, n$ with $(p,q)=1$. 
Then 
$$\dfrac{p}{q}
=[a_1 ,\ldots,a_n]$$
if and only if 
$$ \mat{p}{q}{r}{s}
= \pm \mat{a_1}{1}{1}{0} \cdots \mat{a_n}{1}{1}{0} $$
with $ps-qr=(-1)^n$, where we regard 
$\dfrac{1}{\infty }$ as $0$ and $\dfrac{1}{0}$ as $\infty$. 
\end{lemma}

Then, under the setting of Theorem~\ref{MainThm}, we have the following in the proof of the theorem, which follows from \cite[Proof of Theorem 6]{Kanenobu86}. 

\setcounter{claim}{0}
\begin{claim}
The equations \eqref{eq5} and \eqref{eq6} implies 
\begin{equation*}
\yaji{p}{q}=\pm \yaji{\varepsilon (-1)^{n} \, m^2}{\varepsilon (-1)^{n} \, m k + 1} . 
\end{equation*}
\end{claim}

\begin{proof}
From Eq.\eqref{eq5}, the next can be proved readily. 
\begin{equation*}
\mat{-a_1}{1}{1}{0} \cdots \mat{-a_n}{1}{1}{0}
=\pm (-1)^n \mat{m}{-l}{-k}{n} 
\end{equation*}
with $mn-kl=(-1)^n$.
Then we have the following. 
\begin{align*}
\yaji{p}{q}=&\pm \mat{a_1}{1}{1}{0} \cdots \mat{a_n}{1}{1}{0}\mat{\varepsilon}{1}{1}{0} \mat{-a_n}{1}{1}{0} \cdots \mat{-a_1}{1}{1}{0} \yaji{1}{0} \\
&=\pm \mat{m}{l}{k}{n} \mat{\varepsilon }{1}{1}{0} \left( \mat{-a_1}{1}{1}{0} \cdots \mat{-a_n}{1}{1}{0} \right)^t \yaji{1}{0} \\
&=\pm \mat{m}{l}{k}{n} \mat{\varepsilon }{1}{1}{0} \mat{m}{-k}{-l}{n} \yaji{(-1)^n}{0} \\
&=\pm \mat{\varepsilon m^2}{ - \varepsilon mk + (-1)^n}{\varepsilon mk + (-1)^n }{ - \varepsilon k^2} \yaji{(-1)^n}{0} \\
 =&\pm \yaji{\varepsilon (-1)^{n} m^2}{\varepsilon (-1)^{n} mk+1}.
\end{align*} 
\end{proof}

Then, by the next claim, together with $\varepsilon = \pm 1$ and $p>0$, we have $\frac{p}{q} = [a_{1},\ldots ,a_{n},\varepsilon ,-a_{n},\ldots ,-a_{1}]$.

\begin{claim}
Let $m, k$ be an arbitrary pair of integers. 
Then $(m^2, mk \pm 1)=1$. 
\end{claim}

\begin{proof}
Let $\alpha \in \mathbb{Z}_+$ be the greatest common divisor of $m^2$ and $mk \pm 1$. 
If a prime $p$ divides $\alpha$, then it divides $m^2$, therefore it divides $m$. 
Because it also divides $mk+1$ or $mk-1$, it divides 1, which is absurd. 
We then conclude that $(m^2 , mk \pm 1)=1$. 
\end{proof}

\bibliographystyle{amsalpha}
\bibliography{2brRibbon}

\end{document}